\documentclass[16pt]{article}      % onecolumn (standard format)
\usepackage{graphicx,xcolor,amsthm}
\usepackage{amssymb,mathrsfs,amsmath}
\usepackage[us, 12hr]{datetime}
\usepackage{enumerate}
\newcommand{\labeleq}[1]{\label{eq:#1}}
\newcommand{\refeq}[1]{\text{(\ref{eq:#1})}}
  
\newcommand{\ZZ}{{\mathbb Z}} 
 
\newtheorem{theorem}{Theorem}
\newtheorem{definition}{Definition}
\newtheorem{lemma}{Lemma} 
\newtheorem{remark}{Remark}
\newtheorem{proposition}{Proposition}

\newtheorem{example}{Example}
\usepackage{geometry}    
\graphicspath{../../}
\begin{document}

\title{Frame spectral pairs and exponential bases \thanks{The research of the first author was supported partially by the NSF grant DMS-1720306, and the research of the second author was supported  by an PSC-CUNY award and William P. Kelly Research Fellowship Award.}}
	\author{Christina Frederick \and Azita Mayeli %etc.
	}
%	
%	%\authorrunning{Short form of author list} % if too long for running head
%	
%	\institute{C. Frederick \at
%		Department of Mathematical Sciences\\ New Jersey Institute of Technology \\
%		\email{christin@njit.edu}           %  \\
%		%             \emph{Present address:} of F. Author  %  if needed
%		\and
%		A. Mayeli \at
%		Department of Mathematics\\ City University of New York\\
%		\email{amayeli@gc.cuny.edu}
%	}
%	
%	\date{Received: \today / Accepted: date}
%	% The correct dates will be entered by the editor
%	
%	
	\maketitle
	\begin{abstract}
		Given a domain $\Omega\subset\Bbb R^d$ with positive and finite Lebesgue measure and a discrete set $\Lambda\subset \Bbb R^d$, we say that $(\Omega, \Lambda)$ is a {\it frame spectral pair} if the set of exponential functions $\mathcal E(\Lambda):=\{e^{2\pi i \lambda \cdot x}: \lambda\in \Lambda\}$ is a frame for $L^2(\Omega)$. 
		Special cases of frames include Riesz bases and orthogonal bases.
		In the finite setting 
		$\Bbb Z_N^d$, $d, N\geq 1$, a frame spectral pair can be similarly defined. %(Here, $\Bbb Z_N$ is the cyclic abelian group of order.)
		In this paper we show how to construct and obtain new classes of frame spectral pairs in $\Bbb R^d$ by 
		``adding" a frame spectral pair in $\Bbb R^{d}$  to a frame spectral pair in $\Bbb Z_N^d$. Our construction unifies the well-known examples of exponential frames for the union of cubes with equal volumes.  We also remark on the link between the spectral property of a domain and sampling theory.

%		\keywords{Frames, Riesz bases, exponential bases and  sampling}

	\end{abstract}

	\section{Introduction}
	\label{intro}
	Let $\Omega\subset \Bbb R^d$ be a set with positive and finite  Lebesgue measure, $0<
	|\Omega|<\infty$, and let 
	$\Lambda\subset \Bbb R^d$ be a discrete and countable set. Let $\mathcal E_\Omega (\Lambda)$ denote the set of exponentials, $\mathcal E_\Omega (\Lambda)=\{ e_\lambda: \Omega\to S^1\mid  e_\lambda(x) := e^{2\pi i x\cdot\lambda}\}$. 
	For simplicity, we will drop the subscript $\Omega$ in the sequel and   write $\mathcal E(\Lambda)$. 
	%We say that the set $\mathcal E(\Lambda)$ is a complete set in $L^2(\Omega)$ if for any $u\in L^2(\Omega)$ with $\langle u, e_\lambda\rangle= 0, \forall \lambda\in \Lambda$, we have $u=0$. 
	%We term $(\Omega, \Lambda)$ a {\it \bf frame spectral pair} when the system of
	The set of  exponentials 
	$\mathcal E(\Lambda)$ is a frame for $L^2(\Omega)$ when 
	% If the frame is a Riesz basis, we say that $(\Omega, \Lambda)$ is a Riesz spectral pair. If the frame is an orthogonal basis, we will refer to the pair as a spectral pair. 
	there exist finite  constants $c>0$ and $C>0$ such that for any $u\in L^2(\Omega)$ 
	\begin{align}\label{Frame_ineq_1}
		c \|u\|_{L^{2}(\Omega)}^2 \leq \sum_{\lambda\in \Lambda} |\langle u, e_\lambda\rangle_{L^{2}(\Omega)}|^2 \leq C \|u\|^{2}_{L^{2}(\Omega)}.
	\end{align} 
	The constants $c$ and $C$ are  called {\it lower} and {\it upper frame constants}
	(see the classical references \cite{Ole-book,PeteCaza_art_frame_2000} for the basic properties of frames).   The frame is called {\it tight} if $C=c$. 
	The set of exponentials $\mathcal E(\Lambda)$ is a {\it Riesz sequence} in $L^2(\Omega)$, 
	if there are positive constants $c>0$ and $C>0$ such that for any finite set of scalars 
	$\{c_\lambda\}_{\lambda\in F}$ ($F\subset \Lambda$ finite), the following inequalities hold: 
	\begin{align}\label{RieszB}
		c \sum_{\lambda\in F} |c_\lambda|^2 \leq \left\|\sum_{\lambda\in F} c_\lambda e_\lambda\right\|_{L^{2}(\Omega)}^2 \leq C \sum_{\lambda\in F} |c_\lambda|^2 .
	\end{align} 
	When  only the right inequality in (\ref{RieszB}) holds, then $\mathcal E(\Lambda)$ is called a {\it Bessel sequence}.  A Riesz sequence is called a {\it Riesz basis} if it is complete. Any Riesz basis is a frame with unified frame constants, while the converse is not always true; see e.g. \cite{Ole-book}.
	
	\begin{definition} Given a set $\Omega\subset \Bbb R^d$ and discrete and countable set $\Lambda\subset \Bbb R^d$,  we term 
		$(\Omega, \Lambda)$ a {\it \bf frame spectral pair} when the system of
		exponentials 
		$\mathcal E(\Lambda)$ is a frame for $L^2(\Omega)$.
	\end{definition}
	
	For a given frame pair, when the frame is a Riesz basis, we shall call the pair a {\it Riesz spectral pair}. When it is an orthogonal basis, we shall simply call the pair a {\it spectral pair.}   An interesting question in the frame theory is the following: 
	
	\noindent {\bf Question A.} How can a given frame spectral be {\it modified} to become a Riesz basis or an orthogonal basis? 
	This question is more general, and in its light 
	in this paper we provide an answer to the following question:

	%\noindent {\bf Question A.} {\it Given a frame spectral pair 
	%	in $\Bbb R^d$, how can one construct a new frame spectral pair in $\Bbb R^d$, such that the frame is a Riesz basis or an   orthogonal basis?} \\
	
	\noindent {\bf Question B.} 	 How can current frame spectral pairs, Riesz spectral pairs, or orthogonal bases be  {\it combined}  to form new ones?

	The study of frame spectral pairs in this paper is motivated by the interpolation formula for bandlimited signals with structured spectra and the link between spectral properties of a domain to its sampling properties. 
	
	%Moreover  a domain which does not admit any Riesz basis of exponentials is still unknown and stays as an unsolved problem, and the techniques presented here to induce a new frame spectral pairs by old pairs may shed some light on this problem. {\color{red} Make this a bit stronger, like this gives a new perspective}

	\subsection{Main contributions} To address Question B., 
	the simplest and perhaps the most natural way of creating a new frame spectral pair in $\Bbb R^d$ is by appropriately {\it ``adding"}  two frame spectral pairs in both finite and continuous settings.  
	In this paper,  we will investigate whether  for a given pair  $(\Omega, \Lambda)$ in $\Bbb R^d$,   the pair  
	$(\Omega+A, \Lambda+J/N)$, $A, J\subset \Bbb Z_N^d$, will inherit 
	the analytical property  of the pair  $(\Omega, \Lambda)$.

	Our main findings are as follows.

	\begin{theorem}\label{examples of Frame spectral pairs} Suppose that  %Suppose that $(\Omega_1, \Lambda_{1})$ is frame spectral pair, where  
		$\Omega_1\subset \Bbb R^d$ is a set with positive and finite Lebesgue measure, $\Lambda_1\subset \Bbb R^d$ is a discrete and countable set, and $N\geq 1$ is an integer. Let $A$ be any subset of $\Bbb Z_N^d$ for which the translates of $\Omega_{1}$ by $A$ are disjoint, that is,
		\begin{align}\label{disjointness}
			|\Omega_1+a \cap \Omega_1+a'|=0, \  \forall a,a'\in A, \  a\neq a'.
		\end{align}
		Let  $J\subset \Bbb Z^{d}_{N}$, and define 
		\begin{equation}\label{addition}
			\Omega:= \Omega_1+A, \quad \Lambda:=\Lambda_1 +J/N, 
		\end{equation} 
		where the sum is taken to be Minkowski addition.
		If $(\Omega_1, \Lambda_{1})$ is a frame spectral pair in $\Bbb R^d$ with frame constants $\alpha, \beta$ and $(A, J)$ is a frame spectral pair in $\Bbb Z_N^d$ with frame constants $c, C$, then $(\Omega, \Lambda)$ is also a frame spectral pair  in $\Bbb R^d$ with frame constants $\alpha c, \beta C$ if   for all $\lambda\in \Lambda_1$ we have 
		%the following condition is satisfied: 
		\begin{align}\label{suff-cond} 
			%	\hat\delta_{\lambda}(x)= e^{2\pi i \lambda_1\cdot a} =1, \qquad \forall a\in A, \forall \lambda_{1}\in \Lambda_{1}.
			\hat\delta_{\lambda}\equiv 1 \  \text{on}  \ A,
		\end{align}
		where $\delta_\lambda$    is the Kronecker delta function at $\lambda$ (\ref{Dirac Delta with mass point}).  
		%	
		%	
		%	
		%In other words, $(\Omega, \Lambda)$ is a frame spectral pair in $\Bbb R^d$. 
		%
		%Furthermore, when these conditions are satisfied, the frame constants for
		%$(\Omega, \Lambda)$ are given by 
		%$ \alpha c  $ and   $ \beta C$, where
		%	$\alpha,\beta>0$ are frame constants for the frame spectral pair $(A,J)$, and $c,C>0$ are the frame constants for the frame spectral pair $(\Omega_1, \Lambda_1)$. 
	\end{theorem}

	%Above, $\lambda_1\cdot a$ is the scalar product of two vectors in $\Bbb R^d$. \\ 
	\begin{theorem}\label{examples of Riesz spectral pairs} With the assumptions of 
		Theorem \ref{examples of Frame spectral pairs} and condition (\ref{suff-cond}), $\mathcal E(\Lambda)$ is a Riesz basis for $L^2(\Omega)$ if $\mathcal E(\Lambda_1)$ is a Riesz basis for $L^2(\Omega_1)$ and $\mathcal E(J)$ is a  basis for $\ell^2(A)$. In this case, the frame constants are given as in Theorem \ref{examples of Frame spectral pairs}.
		
		Moreover, when $\Lambda_1$ is a lattice, the dual Riesz basis $\{g_\lambda\}_{\lambda\in \Lambda}$ is given by 
		\begin{align}\label{hlambda} 
			g_{\lambda_1+j/N}(x)  &=
			\left\langle G_{j}, E_{j_{s}} \chi_{\Pi_{\Lambda_{1}}}(x-\cdot)\right\rangle_{l^2(A)} e_{\lambda_1+{j/N}}(x)\qquad  \text{a.e.} \ x\in \Omega.
		\end{align} 
		Here, $\Pi_{\Lambda_{1}}$  is the fundamental domain of the lattice $\Lambda_1$,  $\chi_{\Pi_{\Lambda_1}}$ denotes the indicator function of the domain $\Pi_{\Lambda_{1}}$, and $E_j(z)=e^{2\pi i \frac{z\cdot j}{N}}$ and $G_j$ are given in (\ref{duals}). 
	\end{theorem}
	%According to this theorem, we call $\Omega$ a {\it Riesz spectral set} and $\Lambda$ a {\it Riesz spectrum}. 

	\begin{theorem}\label{example of spectral pair-general} With the assumptions of 
		Theorem \ref{examples of Frame spectral pairs}, $\mathcal E(\Lambda)$ is an orthogonal basis for $L^2(\Omega)$ if $\mathcal E(\Lambda_1)$ is an orthogonal basis for $L^2(\Omega_1)$ and $\mathcal E(J)$ is an orthogonal basis for $\ell^2(A)$ and (\ref{suff-cond}) holds. 	
	\end{theorem}
	
	As an indirect application of Theorem \ref{example of spectral pair-general}, we obtain the following result, which connects a domain's spectral property to sampling and reconstruction.
	
	\begin{proposition}\label{proposition1} 
		Assume that $N>0$, $A, J\subseteq \Bbb Z_N$. Let $\Omega:=[0,1]+A$ and $\Lambda := \Bbb Z +J/N$. Assume that $f\in PW_\Omega$ such that $\{f(\lambda)\}_{\Lambda}$ are predetermined. If $(A, J)$ is a spectral pair in $\Bbb Z_N$, then $f$ can be completely 
		recovered  and 
		$$\hat f(\xi) = (\sharp J)^{-1} \sum_{\lambda\in\Lambda} f(\lambda) e^{-2\pi i \lambda \xi} \quad   a.e.\ \xi\in \Omega.$$ 
	\end{proposition} 
	
	\subsection{Comparison with existing work}
	Regarding Theorem \ref{examples of Frame spectral pairs},   any bounded domain $\Omega$ with Lebesgue measure $0<|\Omega|<\infty$ admits an exponential tight frame $\mathcal E(\Lambda)$ by projecting an exponential orthonormal basis of a cube 
	(containing $\Omega$) onto $\Omega$. In this method, the frame spectrum $\Lambda$ is necessarily a lattice. In Theorem \ref{examples of Frame spectral pairs}, we provide an alternative method for the construction of exponential frames (not necessarily tight) for bounded domains where the frame spectrum is not necessarily a lattice. 
	
	Notice that Theorem \ref{examples of Riesz spectral pairs} illustrates how to construct a Riesz spectrum for the union of domains with equal measure. 
	% Riesz bases form a special class of frames and are known to be the largest and most tractable class of bases. %A basis for a separable Hilbert space is called a {Riesz basis} if it is the image of an orthogonal basis under a bounded and invertible linear operator. % \cite{Young_nonHarmonicBook}. 
	% Riesz bases of exponential functions are of central importance in control theory and Sobolev spaces due to the works of Russel \cite{Russell_82} and Avdonin and Ivanov \cite{Avdonin_Ivanov_1995}. 
	%
	% (See Section \ref{Riesz basis} in the current paper for an equivalent definition of a Riesz basis.) 
	There are many cases where it is known that a set $\Omega$ admits a Riesz basis of exponential functions, such as   multi-tiling (bounded and unbounded) domains in $\Bbb R^{d}$ \cite{GL14,KL_2015,Agora_Gabrielli_2015,Gabriellei_Carbajal_PAMS2018}. Recently, it was established in \cite{Debernardi2019} that any convex polytope which is centrally symmetric and whose faces of all dimensions are also centrally symmetric, admits a Riesz basis of exponentials. For the 
	existence of exponential Riesz bases in other special cases see, e.g., \cite{Nitzan_Kozma_2014,DeCarli2015} and the references therein. 
	While there are known cases where $\Omega$ does not admit an orthogonal basis of exponentials, such as the unit Ball in $\Bbb R^d (d\geq 2)$ (for other cases see e.g.  \cite{IKT01,Fug01}), less is known about Riesz bases of exponentials.
	Finding a domain $\Omega$ that does not admit any Riesz basis of exponentials is still an unsolved problem.

	Below, we point out the difference between our results and some well-known results for the construction of exponential Riesz bases. %, e.g., when $\Omega$ is a finite union of cubes with equal volumes: 
	%An example of exponential Riesz bases were constructed for example in \cite{GL14}, and later by the author  in \cite{KL_2015} with fewer assumptions and a simpler proof. 
	Constructive proofs of the existence  of exponential Riesz bases can be found in \cite{GL14}, and  later in  \cite{KL_2015} with fewer assumptions and a simpler proof.
	For example, in \cite{KL_2015}, the author considers a multi-tiling domain $\Omega$ in $\Bbb R^d$ which multi-tiles the space by a lattice $\Lambda^\perp$ and proves that the domain admits an exponential Riesz basis. More precisely, he obtains  a Riesz spectrum for the domain $\Omega$ using a finite union of translates of the dual lattice $\Lambda$, i.e., $\Lambda+J$. %The all translation vectors in $J$, for example in \cite{KL_2015},  depend on the multi-tiling lattice points $\Lambda^\perp$.
	Here, we choose a Riesz spectral pair $(\Omega, \Lambda)$ (in $\Bbb R^d)$ and a basis pair $(A, J)$ (in $\Bbb Z_N^d)$ and show that  $(\Omega+A, \Lambda+J/N)$ is a Riesz spectral pair in $\Bbb R^d$ if  (\ref{suff-cond}) is satisfied for $A$ and $\Lambda$. 
	%  consider a union of disjoint multi-integer translations of the domain.  The result is not necessarily a multi-tiling set with respect to a lattice.  
	%Then we prove the existence of a Riesz spectrum for this set by taking the union of translations of the spectrum $\Lambda$  by finitely many rational vectors in $J\subset [0,1)^d$.  
	%

	Another well-known example of exponential Riesz bases for a union of co-measurable cubes has been constructed by DeCarli \cite{DeCarli2015} for the union of unit cubes. In this paper, the author takes a finite set of vectors in $\Bbb R^d$ with an arithmetic progression and proves that the union of shifts of $\Bbb Z^d$ by these vectors is a Riesz spectrum for the union of cubes if and only if the evaluation matrix is invertible. Equivalently, for a given finite number of vectors in $\Bbb R^d$, the union of translations of $\Bbb Z^d$ by the vectors is a Riesz spectrum if the matrix is an invertible   Vandermonde matrix. Theorem \ref{examples of Riesz spectral pairs} requires fewer assumptions to establish the existence of Riesz bases for a finite union of unit cubes. More precisely, 
	when $\Omega_1$ is a $d$-dimensional cube, we construct a Riesz spectrum by taking the disjoint union of multi-rational shifts of $\Bbb Z^d$ with the sufficient condition that the matrix is invertible. 
	
	In Theorem \ref{examples of Riesz spectral pairs}, when $\Omega_1$ is a cube in $\Bbb R^d$, the structure of the Riesz spectrum for the union of cubes is similar to the 
	well-known example of sampling and interpolation sequences constructed by 
	Lyubarskii and Seip \cite{Lyubarskii1997} for the union of intervals of equal length in $\Bbb R$, and by Marzo in higher dimensions \cite{Marzo}. In \cite{Lyubarskii1997}, the authors consider a union of $p$ disjoint intervals with equal size and construct a sampling and interpolation sequence for the set using the $p$ shifts of the spectrum of a single interval. Like in the current paper, the sufficient condition is the invertibility of an associated matrix (\ref{associated_matrix}).  Our result in 
	Theorem \ref{examples of Riesz spectral pairs} provides machinery for such constructions with more general domains beyond intervals (or cubes).

	{Regarding Theorem \ref{example of spectral pair-general}}, in recent years, research on 
	orthogonal bases of exponentials has flourished in response to the development of exponential bases in Banach spaces, and in particular, to the growing interest in 
	the Fuglede Conjecture \cite{Fug74}. The Fuglede Conjecture asserts that every domain of $\Bbb R^d$ with positive finite Lebesgue measure admits an orthogonal basis of exponentials if and only if it tiles ${\Bbb R}^d$ by translation.  Although the Fuglede Conjecture is, in general, false, as shown by Tao in one direction \cite{T04} and by Kolountzakis and Matolcsi in the other \cite{KM04,KM06,M05}, it has given rise to active investigations of the connections between orthogonal bases of exponentials and tilings in Euclidean space. The conjecture has been proved affirmative in special cases in various settings (continuous and discrete). See, for example, \cite{GL20,IKT03,IMP17,PhilippBirklbauer,FMV2019,FKS_21}, and the references contained therein. 
	There are many cases where it is known that a domain admits no orthogonal exponential bases. See, for example, \cite{K99,IKT01,Fug01,IK13} and the references contained therein. The conjecture is still open in dimensions $d=1, 2$. However, there are special cases in these dimensions where the conjecture has been proved affirmative (see e.g. \cite{IKT01,Laba}).

	With regard to Theorem \ref{example of spectral pair-general}, we  shall point it out that when $\Omega_1$ is a $d$-dimensional cube, the new spectrum set that we construct in   Theorem \ref{example of spectral pair-general} is different from the well-known spectra in the literature, namely, the one for the union of cubes in $\Bbb R^d (d\geq 5)$ that was  presented by Tao in \cite{T04}. The example is used to disprove the direction ``spectral $\not\rightarrow$ tiling" of the Fuglede Conjecture. In the construction, Tao considers a union of translations of a spectral set in the finite domain $\Bbb Z_p^d$ and lifts it to a higher dimension. Our construction of a new spectral set is the result of adding a spectral set in  $\Bbb Z_p^d$ to a  spectral set in the continuous domain $\Bbb R^d$.

	Another well-known example of a spectral set  was constructed by \L aba in dimension $d=1$ \cite{Laba}. In her paper, \L aba characterizes the spectrum of the union of two co-measurable intervals. When the left endpoints of the intervals are integers, the spectrum of the union coincides with the construction of the spectrum we present in Theorem \ref{example of spectral pair-general}. Here, we construct a spectral set for any ($\geq 2$) union of   co-measurable  intervals. \\

	{\bf Outline.} 
	This paper is organized as follows. After introducing the notations and preliminaries in Section \ref{notations and preliminaries}, 
	we prove 
	Theorem \ref{examples of Frame spectral pairs} in  Section \ref{frames}. In Section \ref{ex:dual Riesz} we prove Theorem \ref{examples of Riesz spectral pairs} and illustrate the explicit structure of biorthogonal dual Riesz bases along with some examples for a special subclass of Riesz spectral pairs and Riesz bases, using the techniques that were developed earlier by the first listed author and Okoudjou in \cite{frederick2020}. Results and examples of dual bases in both the continuous and finite settings appear in Sections \ref{bioDualRiesB} and \ref{examples}. 
	In Section \ref{ex:spectral pair}, we prove Theorem \ref{example of spectral pair-general}  followed by examples of spectral pairs. 
	In Section \ref{sampling_section}  we prove Proposition \ref{proposition1}. %on the link between the  spectral property of a domain and sampling theory. 

	\section{Notations and preliminaries}\label{notations and preliminaries}
	
	Throughout this paper, $\Omega\subset\Bbb R^d$ is a  Lebesgue measurable set with $0<|\Omega|<\infty$, and $\Lambda\subset \Bbb R^{d}$ is  countable and discrete. The inner product of $u, v\in L^2(\Omega)$ is defined by 
	$\langle u, v\rangle_{L^2(\Omega)}= \int_{\Omega} u(x) \overline{v(x)} dx .$
	In the sequel, we shall drop the subscripts and simply write $\langle f, g\rangle$ when the underlying Hilbert space is clear from the context. 
	
	For $f\in L^2(\Bbb R^d)$, we denote by $\hat f$ or $\mathcal{F}(f)$ the 
	Fourier transform of $f$, and it is defined by 
	$\hat f(\xi) = \int_{\Bbb R^d} f(x) e^{-2\pi i x\cdot \xi} dx, \ \xi\in \Bbb R^d,$
	where, 
	$x\cdot \xi=\sum_{i=1}^d x_i\xi_i$ is the scalar products of two vectors. 
	The inverse Fourier transform, denoted by $\mathcal{F}^{-1}(\hat{f})=f$, is then given by 
	$f(x) = \int_{\Bbb R^d} \hat f(x) e^{2\pi i x\cdot \xi} dx, \ x\in \Bbb R^d.$

	Here, and in the sequel, we denote the cardinality of a finite set $F$ by $\sharp F$. 
	For $d, N\geq 1$, $\Bbb Z_N^d$ denotes the $d$-dimensional vector space over the cyclic group $\Bbb Z_N$. 
	For any functions $f, g\in \ell^2(\Bbb Z_N^d)$, the inner product is defined by 
	$\langle f, g\rangle_{\ell^2(\Bbb Z_N^d)} = N^{-d}\sum_{x\in \Bbb Z_N^d} f(x)\overline{g(x)} .$
	In general, for a function $f:\Bbb Z_N^d\to \Bbb C$, $d\geq 1$, the cyclic Fourier transform is given by 
	$\hat f(x) = N^{-d}\sum_{y\in \Bbb Z_N^d} e^{2\pi i y\cdot x/N}$.
	
	A frame spectral pair 
	in the finite setting can be defined in a fashion that is similar to the continuous setting.
	%More precisely, we say a pair $(A, J)$ is a frame spectral pair in $\Bbb Z_N^d$ if the exponentials 
	% $\mathcal E(J) = \{ e^{2\pi i \frac{x\cdot j}{N}}: \ j\in J\}$ is a frame for $\ell ^2(A)$. 
	Definition 1 can also be expressed in terms of matrices, as follows: 
	Let $d, N\geq 1$, and assume that $A, J\subseteq \Bbb Z_N^d$, with $\sharp J \geq \sharp A$.
	Let $\mathcal F$ denote the $N^d\times N^d$ discrete Fourier transform matrix over $\Bbb Z_N$. For a pair 
	$(A, J)$, let $\mathcal {F}_{A,J}$ denote the submatrix of $\mathcal F$ (or {\it evaluation matrix}) with columns indexed by $J$ and rows indexed by $A$, i.e. 
	\begin{align}\label{associated_matrix}
		\mathcal {F}_{A,J} = 
		\begin{bmatrix} 
			\omega^{j\cdot a}
		\end{bmatrix}_{j\in J, a\in A} , 
	\end{align} 
	where $ \omega = e^{-\frac{2\pi i}{N}}$. 
	We say  that   $(A,J)$ is a {\it frame spectral pair in $\Bbb Z_N^d$} if 
	\begin{align}\notag%\label{finite-frame}
		c\|x\|_{\ell^2(A)}^2 \leq \|\mathcal F_{A,J}(x)\|_{\ell^2(J)}^2 \leq C \|x\|_{\ell^2(A)}^2, \quad \forall x\in \Bbb C^{\sharp A}, 
	\end{align} 
	where $0<c<C<\infty$. 
	%It is straightforward that any ``tall" or square submatrix of the  finite Fourier matrix  with columns $A$ and rows $J$ 
	% induces a frame for $\ell^2(A)$.
	The submatrix $\mathcal F_{A,J}$  induces a Riesz basis  (an orthogonal basis) for $\ell^2(A)$ if $\sharp A=\sharp J$ and it is invertible (unitary).

	Notice that the frame spectral property is a symmetry property when the frame is a Riesz basis or an orthogonal basis. This can be easily verified using the associated submatrices of each system.
	More precisely, if $\mathcal E(J)$ is a Riesz (or orthogonal) basis for $\ell^2(A)$, then the submatrix $\mathcal F_{A,J}$ is an invertible (or unitary) matrix. Since the transpose matrix preserves the invertibility and unitary property,  thus  $\mathcal E(A)$ is a Riesz (or orthogonal) basis for $\ell^2(J)$, as claimed. 
	%When the frame is neither a Riesz basis nor an orthogonal, the symmetry property holds in the sense that $\mathcal E(A)$ is a frame sequence in $\ell^2(J)$. The latter means that the set $\mathcal E(A)$ is a frame only for its span vector subspace in $\ell^2(J)$. The exponentials $\mathcal E(A)$ become  a frame for $\ell^2(J)$ if $\sharp A=\sharp J$.
	The notion of symmetry discussed above loses its meaning in the infinite or continuous setting, e.g., $\Bbb Z_N^\infty$ or $\Bbb R^d$. These settings require a well-defined notion of symmetry, which is a challenging problem.
	\vskip.12in 
	
	{\it Notation:} In the sequel, we use $e_\lambda(x)= e^{2\pi i x\cdot \lambda}$ for exponential functions in $\Bbb R^d$, and $E_j(z)= e^{2\pi i \frac{z\cdot j}{N}}$ in the finite setting $\Bbb Z_N^d$.  
	
	\section{Proof of Theorem 1}\label{frames}

	\begin{proof}
		%The 
		%	completeness of $\mathcal E(\Lambda)$ in $L^2(\Omega)$ follows by 
		%	the completeness of the pair $(A,J)$ in $\Bbb Z_p^d$ due to Lemma \ref{completeness}. 
		Assume that $\Omega$ and $\Lambda$ are given as in (\ref{addition})$-$(\ref{disjointness}), and also assume that (\ref{suff-cond}) holds. %To prove that $\mathcal E(\Lambda)$ is a frame for $L^2(\Omega)$, we need to show that 
		% frame inequalities (\ref{Frame_ineq_1})  hold. 
		Let  $u\in L^2(\Omega)$.  Then 
		\begin{align}\notag%\label{eq:1}
			\sum_{(\lambda, j)\in \Lambda_1\times J} \left| \langle u, e_{\lambda+j/N}\rangle_{L^2(\Omega)}\right|^2 &= \sum_{(\lambda, j)\in \Lambda_1\times J} \left| \sum_{a\in A} \langle u, e_{\lambda+j/N}\rangle_{L^2(\Omega_1+a)} \right|^2 \\\notag%\label{U(a)}
			&= \sum_{(\lambda, j)\in \Lambda_1\times J} \left| \sum_{a\in A} \int_{\Omega_1} u(x+a) \overline{e_{\lambda+j/N}(x+a)} dx\right|^2 \\\label{11}
			&= \sum_{(\lambda, j)\in \Lambda_1\times J} \left|  \int_{\Omega_1} \left(\sum_{a\in A}   u(x+a) \overline{e_{j/N}(x+a)}\right) \overline{e_\lambda(x)} dx\right|^2 .
		\end{align}
		The last line follows from (\ref{suff-cond}). For any fixed $j\in J$, since $(\Omega_1, \Lambda_1)$ is a frame spectral pair {and $u_{j}(x):=\sum_{a\in A}   u(x+a) \overline{e_{j/N}(x+a)}$ is in $L^{2}(\Omega_{1})$}, there exists an upper frame constant $\beta>0$ for which %we
		% $$\sum_{\lambda\in \Lambda_1} \left|  \int_{\Omega_1} \left(\sum_{a\in A}   u(x+a) \overline{e_{j/N}(x+a)}\right) \overline{e_\lambda(x)} dx\right|^2 \leq C   \int_{\Omega_1} \left| 
		% \sum_{a\in A}   u(x+a) \overline{e_{j/N}(a)}\right|^2 dx.	 $$
		%obtain the following: 
		\begin{align}\notag%\label{eq:12} 
			(\ref{11})  \leq  \beta \sum_{ j\in J}  \int_{\Omega_1} \left| \sum_{a\in A}   u(x+a) \overline{e_{j/N}(x+a)}  \right|^2dx 
			&= \beta  \int_{\Omega_1} \sum_{ j\in J}  \left| \langle  u(x+\cdot) , {e_{j/N}}\rangle_{\ell^2(A)} \right|^2dx.\\
			&\leq  C \beta \int_{\Omega_1}\|u(x+\cdot)\|^{2}_{\ell^2(A)} dx 
			=C\beta  \|u\|^{2}_{L^2(\Omega)},\notag
		\end{align} 	 
		where  $C$ is the upper frame constant for the frame spectral pair $(A,J)$ in $\Bbb Z_N^d$.  %Applying this to (\ref{eq:12}) we arrive to    
		%		 
		%		 \begin{align}\notag
		%		 (\ref{eq:12}) 
		%		&\leq  C \beta \int_{\Omega_1}\|u(x+\cdot)\|^{2}_{\ell^2(A)} dx 
		%			 =C\beta  \|u\|^{2}_{L^2(\Omega)}. 
		%	\end{align} 
		This completes the proof of the upper frame estimate for the pair $(\Omega, \Lambda)$ with the upper bound constant $C\beta$. The lower bound estimate is obtained similarly. The completeness  of the frame sequence is obtained by  the frame lower bound property. 
	\end{proof}

	We  conclude this section with some examples. \\
	
	\noindent {\it Examples: } 
	\begin{enumerate} 
		\item For any given $A, J\subseteq \Bbb Z_N^d$ and any frame pair $(\Omega_1, \Lambda_1)$ in $\Bbb R^d$, the set of exponentials $\mathcal E(\Lambda)$ is a Bessel sequence in $L^2(\Omega)$ with the Bessel constant $\sharp A\sharp J$. When $\sharp A=1=\sharp J$, the Bessel sequence is a {\it tight frame}, defined as a sequence of functions in $L^{2}(\Omega)$ satisfying (\ref{Frame_ineq_1}) with equal frame constants, i.e., $c=C$. 
		\item    Assume that  the exponentials $\mathcal E(\Lambda_1)$ have a lower frame bound in $L^2(\Omega_1)$.  Then for any $A\subseteq \Bbb Z_N^d$ and $J\subseteq \Bbb Z_N^d$,  
		the exponentials  $\mathcal E(\Lambda)$  also have a lower frame bound   in $L^2(\Omega)$,  where  $\Lambda= \Lambda_1+J$ and $\Omega=\cup_{a\in A} \Omega_1+a$.
		% \item Let $J\subset \Bbb Z_N^d$ be any subset and $A$ be any singleton. Then $\mathcal E(\Lambda)$ is a tight frame for $L^2(\Omega)$ with the frame bound equal to $\sharp J$. 	
	\end{enumerate} 

	The following result is of independent interest.

	\begin{proposition}\label{completeness} 
		Suppose that  %Suppose that $(\Omega_1, \Lambda_{1})$ is frame spectral pair, where  
		$\Omega_1\subset \Bbb R^d$ is a set with positive and finite Lebesgue measure, $\Lambda_1\subset \Bbb R^d$ is a discrete and countable set, and $N\geq 1$ is an integer. Let $A$ be any subset of $\Bbb Z_N^d$ with  $|\Omega_1+a \cap \Omega_1+a'|=0, \forall a'\in A, a'\neq a$, and let $J$ be any subset of  $\Bbb Z_N^d$. Define the pair $(\Omega, \Lambda)$ by
		\begin{equation}
			\Omega:= \Omega_1+A, \quad \Lambda:=\Lambda_1 +J/N.
		\end{equation}
		
		If $\mathcal E(\Lambda_1)$ is a complete set in $L^2(\Omega_1)$,
		then the family $\mathcal E(\Lambda)$ is also a complete set in $L^2(\Omega)$ provided that the condition (\ref{suff-cond})  holds. 
	\end{proposition} 
	\begin{proof}  Let $u\in L^2(\Omega)$ and $\lambda=\lambda_1+j/N\in \Lambda$. Then 
		\begin{align}\label{comp} 
			\langle u, e_\lambda\rangle_{L^2(\Omega)} &= \sum_{a\in A} \langle u, e_{\lambda_1+j/N} \rangle_{L^2(\Omega_1+a)} \\\nonumber
			&= \sum_{a\in A} \langle u(\cdot+a), e_{\lambda_1+j/N}(\cdot+a) \rangle_{L^2(\Omega_1)} \\\nonumber
			&= \langle \sum_{a\in A} u(\cdot+a) e_{j/N}(\cdot+a) , e_{\lambda_1} \rangle_{L^2(\Omega_1)}.
		\end{align} 
		Assume that $\langle u, e_\lambda\rangle_{L^2(\Omega)} = 0$ for all $ \lambda=\lambda_1+j/N\in\Lambda$, (i.e. for all $j\in J$ and $\lambda_1\in \Lambda_1$). By the completeness of the pair $(\Omega_1,\Lambda_1)$, by (\ref{comp}) for all $j\in J$ we obtain 
		
		$$ \sum_{a\in A} u(x+a) e_{j/N}(x+a) = 0 \quad a.e. \ x\in \Omega_1 .$$ 
		Or,
		$$ \sum_{a\in A} u(x+a) e_{j/N}(a) = 0 \quad a.e. \ x\in \Omega_1 .$$ 
		By the completeness of the pair $(A,J)$ in the setting of $\Bbb Z_N^d$, the preceding equality implies that for all $a\in A$, 
		$$ u(x+a) = 0 \quad a.e. \ x\in \Omega_1, $$
		or equivalently, 
		$$u(x)= 0 \quad a.e.\ x\in \Omega_1+a .$$
		This proves that $u=0$ in $L^2(\Omega_1+a)$ for all $a\in A$ and therefore $u=0$ in $L^2(\Omega)$. 
	\end{proof}

	\section{Proof of Theorem 2}\label{ex:dual Riesz} 
	%\subsection{Riesz bases}\label{subsec:RB}
	First we have a motivational result. 
	\begin{proposition}\label{prop_Riesz_seq}
		Assume that $\mathcal E(\Lambda_1)$ is a Bessel sequence in $L^2(\Omega_1)$ with a Bessel constant $C>0$, 
		and let $A, J \subset \Bbb R^d$ be any finite sets. Define  $$\Omega:=\Omega_1+A \ , \ \Lambda:=\Lambda_1 +J .$$ 
		Then for the family of exponentials $\mathcal E(\Lambda)$ the Bessel inequality holds, provided that (\ref{suff-cond}) holds. Indeed,  for any finite set of scalars $\{c_{(\lambda, j)}\}_{(\lambda, j)\in F}$, $F=\Gamma \times I\subset \Lambda_1\times J$, we have 
		\begin{align}\label{Bessel-ineq}
			\left\| \sum_{(\lambda, j)\in F} c_{(\lambda, j)}e_{\lambda+j} \right\|_{L^2(\Omega)}^2 \leq B \sum_{(\lambda, j)\in F} |c_{(\lambda, j)}|^2, 
		\end{align}
		where $B=(\sharp A\sharp J) C $. 
	\end{proposition} 
	\begin{proof}  
		The proof is obtained using the Cauchy-Bunyakovsky-Schwarz inequality:
		\begin{align}\label{Bessel_seq_prep} 
			\left\| \sum_{(\lambda, j)\in F} c_{(\lambda, j)}e_{\lambda+j} \right\|_{L^2(\Omega)}^2&= \left\| \sum_{j\in I}\left( \sum_{\lambda\in \Gamma} c_{(\lambda, j)} e_{\lambda}\right) e_j \right\|_{L^2(\Omega)}^2\\\notag
			& \leq \sum_{a\in A} \left\| \sum_{j\in I}\left( \sum_{\lambda\in \Gamma} c_{(\lambda, j)} e_{\lambda}\right) e_j \right\|_{L^2(\Omega_1+a)}^2 .
		\end{align} 
		For $a\in A$, 
		%let 
		%	\begin{align}\nonumber%\label{theta} 
		%	\theta_a := \sum_{j\in I}\left( \sum_{\lambda\in \Gamma} c_{(\lambda_1, j)} e_{\lambda}\right) e_j \quad \text{on} \quad \Omega_1+a. 
		%	\end{align} 
		%	We have 
		\begin{align}\notag
			%\nonumber\|\theta_a\|_{L^2(\Omega_1+a)}^2 & = 
			\int_{\Omega_1+a} \left| \sum_{j\in I}\left(\sum_{\lambda\in \Gamma} c_{(\lambda,j)} e_\lambda(x)\right) e_j(x)\right|^2 dx 
			& \leq (\sharp I) \int_{\Omega_1+a} \sum_{j\in I}\left|\sum_{\lambda\in \Gamma} c_{(\lambda,j)} e_\lambda(x)\right|^2 dx   & %\text{(By the Cauchy-Schwartz inequality)}
			\\\nonumber
			&  \leq  (\sharp J) \sum_{j\in I} \int_{\Omega_1+a} \left|\sum_{\lambda\in \Gamma} c_{(\lambda,j)} e_\lambda(x)\right|^2 dx & \\\nonumber
			&= (\sharp J) \sum_{j\in I} \left\|\sum_{\lambda\in \Gamma} c_{(\lambda,j)} e_\lambda(x)\right\|_{L^2(\Omega_1+a)}^2 & \\\nonumber 
			&\leq  C (\sharp J)   \sum_{(\lambda, j)\in F} | c_{(\lambda,j)}|^2. & % \text{(By the Bessel property of $\mathcal E(\Lambda_1)$)}\nonumber
		\end{align} 
		The last inequality is obtained using the Bessel property of $\mathcal E(\Lambda_1)$ in $L^2(\Omega_1)$ and the property (\ref{suff-cond}). 
		Now, using the inequality in  (\ref{Bessel_seq_prep}) and after the summing over $A$,  the Bessel inequality (\ref{Bessel-ineq}) holds for $\mathcal E(\Lambda)$ in $L^2(\Omega)$  with a Bessel constant  $B= (\sharp A \sharp J)C$. 
	\end{proof}
	
	Theorem  \ref{examples of Riesz spectral pairs} improves the result of   Proposition \ref{prop_Riesz_seq} in the sense that a Bessel sequence   is a  Riesz sequence or a  Riesz basis  when additional assumptions on $A, J$ and $\mathcal E(\Lambda_1)$ are satisfied.

	\begin{proof}[Proof of Theorem \ref{examples of Riesz spectral pairs}] The completeness of 
		the exponentials $\mathcal E(\Lambda)$ 
		in $L^2(\Omega)$ is due to Theorem    \ref{examples of Frame spectral pairs}.  Indeed, $\mathcal E(\Lambda)$    is a frame for $L^2(\Omega)$ and is therefore complete. It is then sufficient to prove 
		%	The 
		%	completeness of $\mathcal E(\Lambda)$ in $L^2(\Omega)$ follows by Lemma \ref{completeness}. 
		the Riesz inequalities (\ref{RieszB}) for $\mathcal E(\Lambda)$.  For this, let 
		$F$ be a finite index set as in Proposition \ref{prop_Riesz_seq} and $\{c_{(\lambda,j)}\}_{(\lambda,j)\in F}$ be any finite set of scalars. Using the equality in (\ref{Bessel_seq_prep}), we have 
		\begin{align}\notag%\label{Riesz_basis}
			\left\| \sum_{(\lambda, j)\in F} c_{(\lambda, j)}e_{\lambda+j/N} \right\|_{L^2(\Omega)}^2&= \sum_{a\in A} \left\| \sum_{j\in I}\left( \sum_{\lambda\in \Gamma} c_{(\lambda, j)} e_{\lambda}\right) e_{j/N} \right\|_{L^2(\Omega_1+a)}^2 \\\nonumber
			& = \sum_{a\in A} \int_{\Omega_1+a} \left| \sum_{j\in I}\left( \sum_{\lambda\in \Gamma} c_{(\lambda, j)} e_{\lambda}(x)\right) e_{j/N}(x) \right|^2 dx \\\label{last_line}
			& = \int_{\Omega_1} \sum_{a\in A} \left| \sum_{j\in I}\left( \sum_{\lambda\in \Gamma} c_{(\lambda, j)} e_{\lambda+j/N}(x)\right) e_{j/N}(a) \right|^2 dx.
		\end{align}
		The last equality is obtained after applying the assumption that $e_\lambda(a)=1$ for all $\lambda\in \Lambda_1$ and $a\in A$. 
		
		For $j\in I$, define 
		$ d_j(x):= e_{j/N}(x) \sum_{\lambda\in \Gamma} c_{(\lambda, j)} e_{\lambda}(x)$ \text{a.e.} $x\in \Omega_1.$ Then,
		\begin{align}\label{eq:112}
			(\ref{last_line}) = \int_{\Omega_1} \sum_{a\in A} \left| \sum_{j\in I} d_j(x) e_{j/N}(a) \right|^2 dx.  
		\end{align}
		Since $\mathcal E(J)$ is a basis for $\ell^2(A)$, it is a Riesz basis by a result in  \cite{FMS_2019}. Assume that $\beta>0$ is the upper frame constant for the Riesz basis.   Then we have 
		\begin{align}\notag
			(\ref{eq:112})&\leq \beta \int_{\Omega_1}  \sum_{j\in I} |d_j(x)|^2 dx \\\notag
			%&= \beta  \sum_{j\in I} \int_{\Omega_1} |d_j(x)|^2 dx \\\notag
			&= \beta \sum_{j\in I} \int_{\Omega_1} \left|\sum_{\lambda\in \Gamma} c_{(\lambda, j)} e_{\lambda}(x)\right|^2 dx \\\notag
			&= \beta \sum_{j\in I} \left\|\sum_{\lambda\in \Gamma} c_{(\lambda, j)} e_{\lambda}\right\|_{L^2(\Omega_1)}^2  \\\notag
			&\leq \beta  C \sum_{j\in I} \sum_{\lambda\in \Gamma} |c_{(\lambda, j)}|^2.   %\text{(by the Riesz basis property of  $\mathcal E(\Lambda_1)$.)} 
		\end{align} 
		Notice that we obtained the last inequality by the Riesz basis property of $\mathcal E(\Lambda_1)$ in $L^2(\Omega_1)$. 
		A lower Riesz bound can be obtained with a similar calculation. This completes  the proof for the first part of the theorem. The proof of the  second part  is illustrated in Section \ref{multi-tiling-domains}.
	\end{proof}

	\subsection{Explicit form of biorthogonal dual Riesz bases}\label{bioDualRiesB}
	
	It is known that any Riesz basis in a Hilbert space has a biorthogonal dual Riesz basis \cite{Ole-book}.  In the 
	setting of exponential Riesz bases, this statement  reads as follows. 
	\begin{proposition}\label{dual-Riesz} 
		Given any Riesz basis $\mathcal E(\Lambda)$ for $L^2(\Omega)$, there is a unique collection of functions $\{h_\lambda\}_{\lambda\in \Lambda}$ in $L^2(\Omega)$ such that the biorthogonality condition holds: 
		\begin{align}\notag%\label{biorthogonal}
			\langle h_\lambda, e_{\lambda'}\rangle=|\Omega| ~ \delta_\lambda(\lambda') \quad \forall \ \lambda, \lambda'\in \Lambda.
		\end{align} 
		Moreover $\{h_\lambda\}_{\lambda\in\Lambda}$ is a Riesz basis for $L^2(\Omega)$ and any $u\in L^2(\Omega)$ can be represented uniquely as 
		\begin{align}\notag%\label{expansion_in_Riesz}
			u= |\Omega|^{-1} \sum_{\lambda\in \Lambda} \langle u, h_\lambda\rangle e_\lambda = |\Omega|^{-1} \sum_{\lambda\in \Lambda} \langle u, e_\lambda \rangle h_\lambda.
		\end{align}
	\end{proposition} %
	Here, $\delta_t$ is the Dirac function  given by 
	\begin{align}\label{Dirac Delta with mass point} 
		\delta_t (x) = \begin{cases} 1 & \text{if} \ x=t\\
			0 & \text{otherwise} . 
		\end{cases} 
	\end{align}
	The basis $\{h_\lambda\}$ is called {\it biorthogonal dual} Riesz basis for $\mathcal E(\Lambda)$. 
	%In \cite{frederick2020}, the first listed author and Okoudjou 
	%construct the explicit form of biorthogonal Riesz bases of exponentials for a class of multi-tiling sets in $\Bbb R^d$. 
	%We will exploit this result in Section \ref{ex:dual Riesz}. 
	When $\mathcal E(\Lambda)$ is an orthogonal basis for $L^2(\Omega)$,   
	the system of exponentials is self-dual, i.e. $e_\lambda=h_\lambda$, and in (\ref{RieszB}) we have $c=C=|\Omega|^{-1}$.

	\subsubsection{Finite  case: $\Bbb Z_N^d$}
	Let $(A, J)$ be a Riesz spectral pair in $\mathbb{Z}^{d}_{N}$, and suppose that $k=\#A$, that is, $A=\{a_r\}_{r=1}^k$, and $ J=\{j_s\}_{s=1}^k$. This means that the $k\times k$ matrix $\mathcal{F}_{A,J}$ in (\ref{associated_matrix}) is invertible.
	
	For $F\in  \ell^{2}(A)$,
	\[\langle F, E_{j_{s}}\rangle = \sum_{r=1}^{k} F(a_{r})  \omega^{a_r\cdot j_s},\] 
	where $\omega=e^{-2\pi i /N}$. 
	The reconstruction of $F$ in $l^2(A)$ using the Riesz basis $\mathcal E(J)$ %=\{E_{j}(a)=e^{2\pi i a\cdot j/N}\}_{j\in J}$
	is accomplished using the linear system
	\begin{align}\label{Dual_system} 
		\begin{pmatrix}
			\langle F, E_{j_{1}}\rangle\\
			\langle F, E_{j_{2}}\rangle\\
			\vdots\\
			\langle F, E_{j_{k}}\rangle
		\end{pmatrix} =
		\mathcal{F}_{A,J}
		\begin{pmatrix}
			F(a_{1})\\
			F(a_{2})\\
			\vdots\\
			F(a_{k})
		\end{pmatrix}
	\end{align} 
	Assume that $\{G_{j_s}\}_{1\leq s\leq k}$ is the dual in $\ell^2(A)$. 
	Then, we see that by biorthogonality $\langle G_{j_{s}}, E_{j_{s'} }\rangle = \#A \delta_s(s')$, we can recover the dual functions $G_{j_{s}}$. More precisely, for $F= G_{j_{s}}$ in (\ref
	{Dual_system}) we get 
	\begin{align}\notag
		\mathcal{F}_{A,J}^{-1}(k{{\bf e}_{s}}) =
		%
		%\begin{align}
		%\mathcal{F}_{A,J}^{-1}\begin{pmatrix}
		%0\\
		%\vdots\\
		%0\\
		%k\\
		%\langle G_{j_{s}}, E_{j_{s}}\rangle\\
		%0\\
		%\vdots\\
		%0\\
		%\end{pmatrix}  
		\begin{pmatrix}
			G_{j_{s}}(a_{1})\\
			G_{j_{s}}(a_{2})\\
			\vdots\\
			G_{j_{s}}(a_{k})
		\end{pmatrix},
	\end{align}
	where ${\bf e}_{s}$ is a $k$-dimensional vector with $({\bf e}_{s})_{t} = 1$ if $t=s$ and $0$ elsewhere. 
	In summary, the biorthogonal dual Riesz basis can be obtained using the $s$th column of $\mathcal{F}_{A,J}^{-1}$:
	\begin{align}\label{duals}
		G_{j_{s}}(a_{r})  = k (\mathcal{F}_{A,J}^{-1})_{r,s} \qquad r=1, \hdots, k.
	\end{align}
	Notice that by (\ref{duals}), the dual basis is also a set of exponentials only when $\mathcal F_{A, J}$ is a unitary matrix.

	\subsubsection{Continuous case: multi-tiling domains and proof of  (\ref{hlambda})}\label{multi-tiling-domains}
	In this section, we shall illustrate an explicit form of the dual Riesz basis for a class of exponential Riesz bases for a bounded domain $\Omega$ given as in Theorem \ref{examples of Riesz spectral pairs} (although the result extends to a class of unbounded domains). First let us recall a  definition. For a  lattice $\Lambda =M(\Bbb Z^d)$ given by an invertible  $d\times d$ matrix $M$, the volume of $\Lambda$ is defined by $vol(\Lambda)= det(M)$. 
	
	Let $\Lambda_{1}^{*}$ again be a full lattice with a fundamental domain $\Pi_{\Lambda^{*}_{1}}$ and dual lattice $\Lambda_{1}$. Then, let $A=\{a_r\}_{r=1}^k \subset \mathbb{Z}^{d}_{N}$,  and let $J=\{j_s\}_{s=1}^k\subset \mathbb{Z}^{d}_{N}$ be a finite set of vectors such that the $k\times k$ matrix $\mathcal{F}_{A,J}$ in (\ref{associated_matrix}) is invertible.
	Then, $(A, J)$ is a Riesz spectral pair. Assume  that $(\Omega_{1}, \Lambda_{1})$ is also a Riesz spectral pair in $\Bbb R^d$, such that the assumptions (\ref{suff-cond}) and (\ref{disjointness}) hold. Then $\mathcal E(\Lambda)$ is a Riesz basis for $L^2(\Omega)$ as given in Theorem \ref{examples of Riesz spectral pairs}. 
	
	\begin{figure}[t!] 
		\centering
		\includegraphics[scale=1.2]{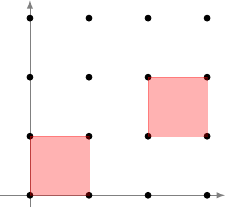} \hspace{1.3cm} 
		\includegraphics[scale=1.2]{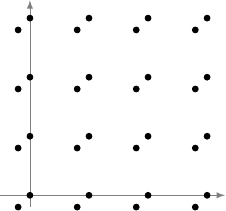} \caption{\footnotesize An example of a multi-tiling set of level $k=2$  in dimension $d=2$ and a corresponding Riesz spectrum.} 
		\label{fig:2dspectralpair}
	\end{figure}
	
	In the given situation, the associated biorthogonal Riesz basis to $\mathcal E(\Lambda)$ with $\Lambda=\Lambda_{1}+\frac{1}{N} J$ has been recently illustrated explicitly and constructively by the first author and Okoudjou in \cite{frederick2020}. The result is presented in \cite{frederick2020} for a more general class of multi-tiling sets, but here we present the special case where $\Omega$ is a multi-rectangle in $\mathbb{R}^{d}$. 
	In this case, 
	%Let $W:=V^{-1}$ denote inverse matrix of $V$. 
	%Define $c_{r,s}=v_{sr} G_{j_{s}}[a_{r}]$.
	% 
	for $\lambda\in\Lambda_{1}$ and $1\leq s\leq k$, the dual basis functions in $L^2(\Omega)$ are given by 
	\begin{align}\notag\label{hlambda} 
		g_{\lambda+j_{s}/N}(x) %&=
		%e_{\lambda+j_{s}/N}(x) 
		%\sum_{r=1}^{k}v_{sr} G_{j_{s}}[a_{r}] 
		%\delta_{a_r} \ast 
		%\chi_{\Pi_{\Lambda_{1}}}(x), \quad a.e. \ x\in \Omega.\\
		%&=
		%e_{\lambda}(x) \left[
		%\sum_{r=1}^{k} G_{j_{s}}[a_{r}]  e^{2\pi i (x-a_r) \cdot j_{s}/N} 
		%\chi_{\Pi_{\Lambda_{1}}}(x-a_{r})\right]\\
		&=
		e_{\lambda+{j_{s}/N}}(x) \langle G_{j_{s}}, E_{j_{s}} \chi_{\Pi_{\Lambda_{1}}}(x-\cdot)\rangle_{l^{2}(A)}\qquad  \text{ a.e.} \ x\in \Omega.
	\end{align} 
	Here, $\chi_{\Pi_{\Lambda_{1}}}$ denotes the indicator function of the domain $\Pi_{\Lambda_{1}}$. For an illustration of a multi-tiling set in dimension $d=2$ see Figure \ref{fig:2dspectralpair}. 
	
	\begin{remark}
		If the system is self-dual, then this formula implies that $(\mathcal{F}_{A,J})^{-1}_{rs}=\frac{1}{k} e^{2\pi i a_r\cdot j_{s}/N}=
		\frac{1}{k}\bar\omega^{a_r\cdot j_s}$. (Recall that $k=\sharp A$.) In this case, 
		$(\Omega, \Lambda)$ is a spectral pair if and only if $(\mathcal{F}_{A,J})^*(\mathcal{F}_{A,J})=kI$ meaning that $\mathcal{F}_{A,J}$ is a (log) Hadamard matrix \cite{KM06}.
	\end{remark}
	
	\begin{remark} A special case is a 1-tiling with respect to the full lattice  
		$\Lambda_{1}^{*}$  with a fundamental domain $\Omega_{1} = \Pi_{\Lambda_{1}}$ and dual lattice $\Lambda_{1}$. In this case, for any $f\in L^2(\Omega_1)$ we have 
		\begin{align}\notag
			\sum_{\lambda\in\Lambda_{1}}\langle f, e_{\lambda} \rangle e_{\lambda}(x) %&= \sum_{\lambda\in\Lambda_{1}} [\mathcal{F}^{-1}(f)](-\lambda) e_{\lambda}(x)=vol(\Lambda_{1}) \sum_{\lambda^{*}\in\Lambda_{1}^{*}} \mathcal{F}\left\{ [\mathcal{F}^{-1}(f)](-\lambda) e_{\lambda}(x)\right\} (\lambda^{*})\\
			\notag
			&=vol(\Lambda_{1}) \sum_{\lambda^{*}\in\Lambda_{1}^{*}}f (x-\lambda^{*}) = vol(\Lambda_{1}) f(x).
		\end{align}
		Here we used the Poisson summation formula and the fact that $\Omega_{1}$ is a fundamental domain of $\Lambda_{1}^{*}$ (the only value of $\Lambda^{*}$ that gives a nonzero summand is $\lambda^{*}=0$).
		This means that for $f\in  L^{2}(\Omega_{1})$,
		\[f(x) = \frac{1}{vol(\Lambda_{1})}\sum_{\lambda\in\Lambda_{1}}\langle f, e_{\lambda} \rangle e_{\lambda}(x) = \sum_{\lambda\in\Lambda_{1}} \langle f, e_{\lambda}\rangle g_{\lambda}(x).\]
		The biorthogonal dual Riesz basis is then given by 
		\[g_{\lambda}(x) = \frac{1}{vol(\Lambda_{1})} e_{\lambda}(x).\]

	\end{remark} 
	
	\subsection{Examples}\label{examples}
	In this section, we provide few examples of a Riesz spectral pair as well as the biorthogonal dual basis (\ref{hlambda}) in dimensions $d=1,2$. 
	Let $\Lambda_{1}=\ZZ^2$, and let  $\Pi_{\Lambda_{1}}=\mathcal Q_2$ be the unit square. For a pair of distinct multi-integers $a_1,a_2\in \Lambda_{1}$, let 
	$$\Omega:=( \mathcal Q_2+a_1)\cup (\mathcal Q_2+a_2).$$
	The set $\Omega$ multi-tiles $\Bbb R^2$ with the lattice $\Bbb Z^2$ at level $k=2$. 
	Let $N>2$ be an integer and $j_1, j_2$ be any two vectors in $\Bbb Z^2$ such that 
	the following matrix is invertible:

	\begin{align}\labeleq{evaluation-matrix}
		V=\mathcal{F}_{A,J}:=\begin{pmatrix}
			\omega^{a_1\cdot j_1} & \omega^{a_2\cdot j_1} \\
			\omega^{a_1\cdot j_2} & \omega^{a_2\cdot j_2} 
		\end{pmatrix}.
	\end{align}
	
	Take 
	\begin{align}\notag
		W=(\mathcal{F}_{A,J})^{-1}=\frac{1}{\text{det}(V)}\begin{pmatrix}
			\omega^{a_2\cdot j_2} & -\omega^{a_2\cdot j_1} \\
			-\omega^{a_1\cdot j_2} & \omega^{a_1\cdot j_1} 
		\end{pmatrix}.
	\end{align}
	Since $\text{det}(\mathcal{F}_{A,J})\neq 0$, then the following system of exponential functions is a Riesz basis for $L^2(\Omega)$:
	
	\begin{align}\notag
		\{e_{n+j_{1}/N}(x)\}_{n\in \ZZ^d}\cup \{e_{n+j_{2}/N}(x)\}_{n\in \ZZ^d}, 
	\end{align}
	and by 
	(\ref{hlambda})
	the dual basis functions are given by 
	\begin{align}\notag
		\{g_{n+j_{1}/N}(x)\}_{n\in \ZZ^d}\cup \{g_{n+j_{2}/N}(x)\}_{n\in \ZZ^d}, 
	\end{align}
	where 
	\begin{align}\notag
		g_{n+j_{1}/N}(x)=2 e_{n+j_{1}/N}(x) \left( {v_{11}} w_{11}\chi_{\mathcal Q_2}(x-a_1)+ {v_{12}} w_{21}\chi_{\mathcal Q_2}(x-a_2)\right) \quad n\in \ZZ^2, \end{align}
	and 
	\begin{align}\notag g_{n+j_{2}/N}(x)=2 e_{n+j_{2}/N}(x)\left(v_{21} w_{12}\chi_{\mathcal Q_2}(x-a_1)+v_{22} w_{22}\chi_{\mathcal Q_2}(x-a_2)\right) \quad n\in \ZZ^2. 
	\end{align}
	
	\begin{example}(case $d=1$)\label{ex1} Following the notation in Section \ref{examples}, 
		let $\{a_1=0,a_2=2\}$,  let $\{j_1=0, j_2= 1\}$, and let $N>2$ be  any integer. Then the matrix $V$ in \refeq{evaluation-matrix} is invertible, so $(\Omega, \Lambda)$ is a Riesz spectral pair, where
		$$\Omega= [0,1] \cup [2,3]\ , \ \Lambda = \Bbb Z \cup \Bbb Z+ 1/N.$$
		
		Then the exponential Riesz basis constructed above coincides with the exponential Riesz basis given as in Theorem \ref{examples of Riesz spectral pairs}, when $\Omega_1=[0,1], \Lambda_1=\Bbb Z$, $A=\{0, 2\}$, $J=\{0, 1\}$, and $N>2$. By the symmetry property, we obtain that $\tilde\Omega=[0,2]$ and $\tilde\Lambda =\Bbb Z\cup \Bbb Z+2/N$ is also a Riesz spectral pair.    It can be readily verified that the system is an orthogonal basis only for $N=4$; see Theorem \ref{example of spectral pair-general}. In this case, $V$ is a unitary matrix with $V^*V=2I$ and the basis is self-dual. 
	\end{example} 
	
	\begin{example}(case $d=2$)
		As a concrete example in $d=2$, let $N=4$, $a_1=(0,0)$, $a_2=(2,0)$, $j_1=(0,0)$, and $j_2=(1,0)$. For this choice, the matrix $V$ in \refeq{evaluation-matrix} is invertible, so the following pair  $(\Omega, \Lambda)$ builds a Riesz spectral pair: 
		$$\Omega= \mathcal Q_2 \cup \mathcal Q_2+(2,0)\ , \ \Lambda = \Bbb Z^2 \cup \Bbb Z^2+ (1/4,0) .$$ 
		By the symmetry property, for $\tilde \Omega= [0,2]\times [0,1]$ and $\tilde\Lambda= \Bbb Z^2\cup \Bbb Z^2+(1/2,0)= 2^{-1}\Bbb Z\times \Bbb Z$, the pair $(\tilde\Omega, \tilde\Lambda)$ is also  a Riesz spectral pair, as we expected.
	\end{example} 
	
	\begin{example}(case $d>1$) Let  $\Omega$ be the convex hull of  a finite number of points in $\Bbb R^d$ which is centrally symmetric.  Moreover, assume that all faces of  $\Omega$ in all dimensions are also centrally symmetric. Then by Theorem 1.1 in \cite{DN_20}, $L^2(\Omega)$ has a Riesz basis of exponentials $\mathcal E(\Lambda)$.   Assume that $(A,J)$ is a pair in $\Bbb Z_N^d$ whose evaluation matrix is invertible. Then the pair $(\Omega+A, \Lambda+J)$ is a Riesz spectral pair in $\Bbb R^d$  if (\ref{suff-cond}) holds. % for the set $A$ and $\Lambda$.  
	\end{example} 
	
	\section{Proof of Theorem \ref{example of spectral pair-general} and 
		Proposition \ref{proposition1}}\label{ex:spectral pair}   
	
	% It is custom to refer to a 
	% frame spectral pair $(\Omega, \Lambda)$   as {\it spectral pair} if the frame 
	% $\mathcal E(\Lambda)$ is 
	% an orthogonal basis for $L^2(\Omega)$, i.e., exponentials $e_\lambda$ are mutually orthogonal: 
	% $$\langle e_\lambda, e_{\lambda'} \rangle= 0 \quad \forall \lambda, \lambda'\in \Lambda, \lambda \neq \lambda'.$$ 

	Given positive integers $N, d\geq 1$, and sets $A, J\subseteq \Bbb Z_N^d$ of the same cardinality, we say  that the family of exponentials  $\mathcal E(J)=\{ E_j(z)=e^{2\pi i j\cdot z /N}\}_{j\in J}$ is an orthogonal basis for  $\ell^2(A)$ 
	if  the elements of  $\mathcal E(J)$ are mutually  orthogonal on $A$; that is 
	\begin{equation}\label{orthogonality} 
		\langle E_j, E_{j'}\rangle_{\ell^2(A)} = \sum_{a\in A} e^{2\pi i (j-j')\cdot a/N} = 0 \quad \forall j, j'\in J, \ j\neq j'.
	\end{equation}

	%As a result of (II),
	%the set of exponentials $\mathcal E(J)$ is a complete set over $A$.
	% As we pointed out in Section \ref{FF}, $\mathcal E(J)$ is an orthogonal basis for $\ell^2(A)$ 
	%if the associated Fourier submatrix (\ref{associated_matrix}) is a unitary matrix. 
	%Trivial cases include $A=J=\Bbb Z_N^d$ or when $A$ is a singleton set. 

	Notice that the orthogonality relation (\ref{orthogonality}) can also be expressed as 
	$$\hat\chi_A(j-j') \equiv   \sum_{a\in A} e^{2\pi i (j-j') \cdot a/N} = 
	0, \quad \forall j\neq j'$$
	where $\hat\chi_A$ is the discrete Fourier transform of the  indicator function $\chi_A$. 
	%
	
	%Now we present the proof of Theorem \ref{example of spectral pair-general}. 
	
	\begin{proof}[Proof of Theorem \ref{example of spectral pair-general}]
		To prove mutual  orthogonality of the exponentials $\mathcal E(\Lambda)$ in $L^2(\Omega)$,  let $\lambda_1+j_1/N$ and $\lambda_2+j_2/N$ be two distinct vectors in $\Lambda$. We have the following: 
		\begin{align}\nonumber 
			\langle e_{\lambda_1+j_1/N}, e_{\lambda_2+j_2/N}\rangle_{L^2(\Omega)} &= \sum_{a\in A} \langle e_{\lambda_1+j_1/N}, e_{\lambda_2+j_2/N}\rangle_{L^2(\Omega_1+a)}\\\label{assumption1}
			&= \left(\sum_{a\in A} e_{\lambda_1+j_1/N}(a) e_{\lambda_2+j_2/N}(-a) \right) 
			\langle e_{\lambda_1+j_1/N}, e_{\lambda_2+j_2/N}\rangle_{L^2(\Omega_1)}\\\label{assumption}
			&= \left(\sum_{a\in A} e_{j_1/N}(a) e_{j_2/N}(-a) \right) 
			\langle e_{\lambda_1+j_1/N}, e_{\lambda_2+j_2/N}\rangle_{L^2(\Omega_1)}\\\label{last-line1}
			&= \langle E_{j_1}, E_{j_2}\rangle_{l^2(A)} \langle e_{\lambda_1+j_1/N}, e_{\lambda_2+j_2/N}\rangle_{L^2(\Omega_1)}.
			%&= \langle e_{\lambda_1+j_1/N}, e_{\lambda_2+j_2/N}\rangle_{\ell^2(A)} \langle e_{\lambda_1+j_1/N}, e_{\lambda_2+j_2/N}\rangle_{L^2(\Omega_1)}
		\end{align}
		
		To pass from (\ref{assumption1}) to (\ref{assumption}) we used the assumption that $\hat\delta_{\lambda}(a)=1$ for all $a\in A$ and $\lambda\in \Lambda_1$.
		If $j_1\neq j_2$, by the orthogonality of $\mathcal E(J)$, the first inner product in (\ref{last-line1}) is zero. 
		If $j_1= j_2$, then by the assumption we must have $\lambda_1\neq \lambda_2$,  and by the orthogonality of $\mathcal E(\Lambda_1)$ we get 
		$\langle e_{\lambda_1+j_1/N}, e_{\lambda_2+j_2/N}\rangle_{L^2(\Omega_1)} = \langle e_{\lambda_1}, e_{\lambda_2}\rangle_{L^2(\Omega_1)} =0.$ The completeness of $\mathcal E(\Lambda)$ in $L^2(\Omega)$ is obtained by a density condition (see \cite{Duffin_Shaeffer_52}), or directly by  Theorem \ref{examples of Frame spectral pairs} since $\mathcal E(\Lambda)$ is a frame for $L^2(\Omega)$. 
		
	\end{proof}

	\begin{remark} Note due to the symmetry property of spectral pairs in finite settings, by the assumption of Theorem \ref{example of spectral pair-general}, the domain 
		$\Omega=\Omega_1 + J$  admits an orthogonal basis of exponentials $\mathcal E(\Lambda)$ with  $\Lambda= \Lambda_1+A/N$, provided that (\ref{disjointness}) and (\ref{suff-cond}) are satisfied. 
	\end{remark}
	
	\begin{remark}\label{cor} Let $(\Omega_1, \Lambda_1)$ be a spectral pair. Assume that $A\subset \Bbb Z^d$ is finite such that (\ref{disjointness}) and (\ref{suff-cond}) are satisfied,  and $B\subset \Bbb R^d$ is  a finite set of rational numbers. Assume that $\sharp A=\sharp B=d$. Then there exists an integer number $N\geq 1$ such that $NB\subset \Bbb Z^d$. Consider the $d\times d$ matrix $V=(v_{a,b})_{a\in A, b\in B}$ with entries given  by 
		$$v_{a,b} = e^{-2\pi i a\cdot b} .$$
		If $V$ is a unitary (or an invertible) matrix, then $\mathcal E(\Lambda)$ is an orthogonal basis (or a Riesz basis) for $L^2(\Omega)$, where 
		\begin{align}\notag%\label{pairs}
			\Omega=\Omega_1+A, \ \Lambda=\Lambda_1+B .
		\end{align}
		%	If $V$ is an invertible matrix, then the pair in (\ref{pairs}) is Riesz spectral. 
		
		This can be  readily obtained by Theorem  \ref{example of spectral pair-general} since $(A, NB)$ is a spectral pair in $\Bbb Z_N^d$. The proof of the  Riesz basis is  due to Theorem  \ref{examples of Riesz spectral pairs} when $V$ is an invertible matrix. 
		
	\end{remark} 
	
	\subsection{Examples of spectral pairs}\label{examples_spect_pairs}  
	The
	first example shows that the sufficient condition (\ref{suff-cond}) is also a {\it necessary condition} in Theorem \ref{example of spectral pair-general}.
	
	\begin{example} Let $\Omega_1=[0,2]$.  Take  $N\in 6\Bbb Z$, $N>0$,  and    $A=\{0,3\}$, $J=\{0,\alpha=N/6\}$. 
		Then $(A, J)$ is a spectral pair in $\Bbb Z_N^d$, while $(\Omega, \Lambda)$ fails to be a spectral pair in $\Bbb R^d$. Indeed, for   $\Lambda=2^{-1}\Bbb Z \cup 2^{-1}\Bbb Z+6^{-1}$, the exponentials $\mathcal E(\Lambda)$ are not mutually orthogonal on $\Omega=[0,2]\cup [3,5]$.  Otherwise we would have had 
		$\Lambda\subseteq \Lambda-\Lambda \subseteq Z_\Omega$, where 
		$$Z_\Omega=\{0\} \cup \{\xi\in \Bbb R: ~ \hat \chi_\Omega(\xi)=0\}.$$  On the other hand,  
		$Z_\Omega= 2^{-1}\Bbb Z \cup 3^{-1}\Bbb Z$, and $\Lambda$ is not contained in $Z_\Omega$. This  is a contradiction,  thus the functions in   $\mathcal E(\Lambda)$ are not orthogonal on $\Omega$. 
		%Indeed, it can be  readily check that 
		%$\langle e_{1/6}, e_{1/2n}\rangle_{L^2(\Omega)}\neq 0$. 
	\end{example}

	\begin{example}\label{example of spectral pair} Let $N\geq 1$, and assume that $(A,J)$ is a spectral pair in $\Bbb Z_N$. Take 
		$\Omega:= \cup_{a\in A} [a,a+1). $
		Then, by Theorem \ref{example of spectral pair-general}, the pair 
		$(\Omega, \Lambda)$ is spectral in $\Bbb R$ for 
		$\Lambda=\Bbb Z+J/N$. 
	\end{example}

	\begin{example}[$d=1$]\label{example:d=1} Let $N=4$, $A=\{0,2\}$ and $J=\{0,1\}$. Define the $2\times 2$ {\it evaluation matrix}
		$$H= \frac{1}{\sqrt{2}} (\omega^{a\cdot j})_{a\in A , j\in J} .$$ 
		Then, $H$ is a unitary matrix, that is, $H^*H=I$. This proves that 
		$(A, J)$ is a spectral pair in $\Bbb Z_4$. By Theorem \ref{example of spectral pair-general}, for $\Omega= [0,1] \cup [2,3]$ and $\Lambda= \Bbb Z \cup \Bbb Z+1/4$, the pair $(\Omega, \Lambda)$ is a spectral pair $\Bbb R$. This example is illustrated in Figure \ref{fig:1dspectralpair}. 
		\begin{figure}[t] 
			\centering
			\includegraphics[scale=1.2]{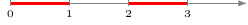} 
			\includegraphics[scale=1.2]{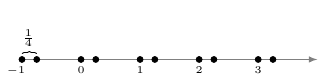}
			\caption{\footnotesize The figure illustrates $\Omega=[0,1]\cup [2,3]$ and its spectral set $\Lambda= \Bbb Z \cup \Bbb Z+1/4$ in $d=1$.} 
			\label{fig:1dspectralpair}
		\end{figure}
	\end{example} 
	By the symmetry property, the pair $(J,A)$ is also a spectral pair. Therefore, using the result of Theorem \ref{example of spectral pair-general} we obtain the well-known spectral pair $(\tilde\Omega=[0,2], \tilde\Lambda=2^{-1} \Bbb Z)$. 
	
	\begin{example}[$d=2$] 
		To construct a spectral pair in higher dimensions, one possible technique is using the Cartesian product as in the results of Jorgensen and Pedersen in \cite{JorgPed99}, Section 2. %Here, the authors prove the existence of a spectral pair constructively and explicitly in higher dimensions using the Cartesian products of the spectral sets. 
		See also \cite{JoPe92,JoPe94} for more on the construction and `closeness' of spectral pairs in Cartesian settings. An example of a spectral pair in dimension $d=2$ 
		is depicted in Figure \ref{fig:2dspectralpair2}. This example is constructed from the self Cartesian product of the spectral pair given  in Example \ref{example:d=1}. 
		%for cartesian product of the example that we constructed in \ref{example:d=1}. 
		{
			\begin{figure}[t!] 
				\centering
				\includegraphics[scale=1.2]{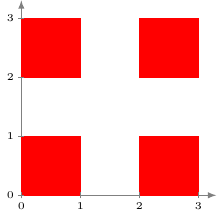} \hspace{1.3cm} 
				\includegraphics[scale=1.6]{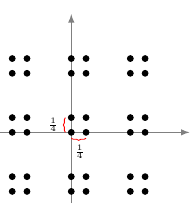} \caption{\footnotesize An example of a spectral set with its spectrum in $d=2$} 
				\label{fig:2dspectralpair2}
			\end{figure}
		}
	\end{example}

	\subsection{Spectral set and function recovery}\label{sampling_section}

	In this section, we remark on the link between spectral pairs and recovery problem, known as Shannon Sampling theorem. 
	% In this section, we explain the link between a spectral set, its spectrum and Shannon-Sampling recovery theorem, which is  explained 
	%in Corollary \ref{recovery}. 
	First, we need the following general result for any pair in $\Bbb Z_N$. 
	
	\begin{lemma}\label{lemma:dist}
		Define $\Omega=  [0,1]+A$ and $\Lambda= \Bbb Z+J/N$ for any sets $A, J \subset \Bbb Z_{N}$, and suppose that $f$ is a function in the Paley-Wiener space $PW_{\Omega}= \{ f\in L^2(\Bbb R)\mid \hat{f}(\xi)=0 \text{ a.e. } \xi \not \in \Omega\}$.  Then, the distribution 
		\begin{align}\label{Fs-function1}
			F_s(t)=  \sum_{\lambda\in\Lambda} f(\lambda) \delta_{\lambda}(t), \ t\in \Bbb R
		\end{align}
		has a Fourier transform given by
		\begin{align}
			\hat F_s(\xi) =\sum_{k\in\Bbb Z} \hat \chi_J(k) \hat f(\xi-k)\quad a.e. \  \xi \in \Omega.
			\label{Fs-function1}
		\end{align}
		
	\end{lemma}
	
	\begin{proof} We express the distribution (general function)  $F_s$ as 
		\begin{align}\label{Fs-function2}
			F_s(t)= f(t) P(t) 
		\end{align}
		where 
		$P\in \mathscr{S}'(\Bbb R)$ is a distribution, known as a {\it  ($\Lambda$-sampling) pattern function}, given by 
		\begin{align}\notag%\label{LambdaSamplingPattern}
			P(t):= \sum_{\lambda\in\Lambda} \delta_\lambda(t) = \sum_{n\in \Bbb Z} \delta_n(t) \ast \sum_{j\in J} \delta_{j/N}(t), \ t\in \Bbb R. 
		\end{align} 
		
		An example of such pattern is illustrated in Figure \ref{fig:1dpattern}. 
		{\begin{figure}[t] 
				\centering
				\includegraphics[width=.5\textwidth]{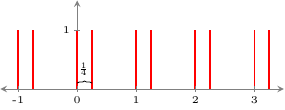} 
				\caption{\footnotesize The figure illustrates the sampling pattern $\Lambda=\Bbb Z\cup \Bbb Z+1/4$ in $d=1$.} 
				\label{fig:1dpattern}
		\end{figure} } 	
		By applying the Fourier transform to  (\ref{Fs-function2}) we obtain 
		\begin{align}\label{convolution_of_FT} 
			\hat F_s = \hat f \ast \hat P, 
		\end{align}
		where
		\begin{align}\label{FT_of_pattern}
			\hat P(\xi) =\mathcal F (\sum_{n\in \Bbb Z} \delta_n)(\xi) \mathcal F(\sum_{j\in J}\delta_{j/N})(\xi), \ \xi \in \Bbb R.
		\end{align}
		%(Here, $\mathcal F$ denotes the Fourier transform.)
		%Next we calculate the Fourier transform of each term in (\ref{FT_of_pattern}): by
		By applying the Poisson summation formula to the first term  of (\ref{FT_of_pattern}), we get  
		\begin{align}\label{*}
			\mathcal F (\sum_{n\in \Bbb Z} \delta_n)(\xi)= \sum_{n\in\Bbb Z} e^{-2\pi i n\xi}=\sum_{k\in \Bbb Z} \delta_k(\xi), \ \xi\in \Bbb R
		\end{align}
		For the second term in (\ref{FT_of_pattern}), in terms of distributions we have 
		\begin{align}\label{**}\mathcal F(\sum_{j\in J}\delta_{j/N})(\xi)= \sum_{j\in J} e^{-2\pi i j\xi/N} , \ \xi \in \Bbb R. 
		\end{align}
		Define the function\footnote{Note that the definition of this general function over $\Bbb R$, also called the {\it symbol} of $J$, coincides with the Fourier transform of the characteristic function $\chi_J$ over the cyclic group when the domain is restricted to $\Bbb Z_N$ up to a constant.} $\hat\chi_J(\xi):=\sum_{j\in J} e^{-2\pi i j\xi/N}$. By (\ref{*}) and plugging (\ref{**})
		in (\ref{FT_of_pattern}), we obtain 
		$$\hat P(\xi)= \hat\chi_J(\xi) \sum_{k\in \Bbb Z} \delta_k(\xi) = \sum_{k\in \Bbb Z} \hat\chi_J(k) \delta_k(\xi).$$
		By substituting this expression in (\ref{convolution_of_FT}) we obtain 
		\begin{subequations}\label{F_s}
			\begin{align}
				\hat F_s(\xi)&= \hat f \ast\hat P(\xi) \\
				&=\left( \hat f\ast \sum_{k\in\Bbb Z} \hat \chi_J(k) \delta_k \right)(\xi)\\
				&= \sum_{k\in\Bbb Z} \hat \chi_J(k) \left(\hat f \ast \delta_k\right) (\xi)\\\label{last_summation}
				&= \sum_{k\in\Bbb Z} \hat \chi_J(k) \hat f(\xi-k)\quad \forall \xi \in \Omega. 
			\end{align} 
		\end{subequations} 
		This completes the proof\footnote{By (\ref{last_summation}), 
			the Fourier transform of $F_s$ is equal to a sum of translated copies of the Fourier transform of $f$ on $k$-shifts of $\Omega$ multiplied with coefficients $\hat \chi_J(k)$. The theorem proves that the Fourier transform of $F_s$ over $\Omega$ is the exact Fourier transform of $f$ up to some constant, and the translations do not overlap. In the language of signal processing, this means that the {\it aliasing term} is zero.}.
	\end{proof}
	Note that Lemma \ref{lemma:dist} holds for any domain $\Omega\subset \Bbb R^d$, $d\geq 1$,  with finite and positive measure and any $\Lambda=\Lambda_1+J/N$, $J\subset \Bbb Z_N^d$, such that  the Poisson summation holds for $\Lambda_1$.

	We are now ready to prove Proposition \ref{proposition1}. 
	%As an application of the previous lemma, we have the following result. 
	%\begin{theorem} 
	%	Given a spectral pair $(A, J)$ in $\Bbb Z_N$, let $\Omega$ and $\Lambda$  be $\Omega:=  [0,1]+A$ and $\Lambda= \Bbb Z+J/N$. Let $f\in L^2(\Bbb R)$ with $\hat f$ supported in $\Omega$.   Let $F_s$ be as in (\ref{Fs-function1}). Then 
	%	$$\hat F_s(\xi)= (\sharp J) \hat f(\xi),  \ a.e.\  \xi\in \Omega.$$
	%\end{theorem} 
	\begin{proof}[Proof of Proposition \ref{proposition1}]
		Let $F_s$ be as in (\ref{Fs-function1}).  
		Then 
		$$\hat F_s(\xi)= (\sharp J) \hat f(\xi),  \ a.e.\  \xi\in \Omega.$$
		Indeed, we need to show that all the terms in (\ref{last_summation}) are zero except for $k=0$. We do this by considering different cases for the values of $k$. 
		
		\begin{enumerate}
			\item[\textbf{I:}] If $k=0$, we have %
			$\hat \chi_J(k)\hat f(\xi-k) = \sharp J \hat f(\xi).$ 
			\item[\textbf{II:}] If $k\neq 0$, the summation is known as {\it aliasing term} and we consider two cases as well: 
			\begin{enumerate}
				\item[(a)] {$k\in A-A$, $k\neq 0$. } Then, 
				$k=a-a'$ for some distinct $a,a'\in A$. 
				By the symmetry 
				property of the spectral pairs in $\Bbb Z_N$ (and more generally in $\Bbb Z_N^d$), $A$ is a spectral set for the set $J$ and we obtain 
				$\hat\chi_J(k)=\hat \chi_J(a-a')=0.$ 
				\item[(b)] $k\in \Bbb Z\backslash \{A-A\}$, $k\neq 0$. 
				In this case, we claim that $|\Omega\cap \Omega+k|=0$. If not, then there must be $a, a'\in A$ such that 
				$|[a+k, a+k+1)\cap [a', a'+1)|\neq 0 .$
				Since $a\in \Bbb Z$ and each interval has length unit $1$, then we must have $k= a'-a$, which is a contradiction. Thus the sum over all $k\not\in A-A$ is equal to zero. 
			\end{enumerate}
		\end{enumerate}
		% Indeed, let 
		%$\xi\in \Omega\cap \Omega+k$. Then there must be an $a\in A$ such that $\xi\in [a+k, a+k+1)$. We also know $\xi\in [a', a'+1)$ for some $a'\in A$. Since $a, a', k$ are integers, then we must have $k= a'-a$, which is a contradiction. Thus the sum over all $k\not\in A-A$ is equal to zero. 
		%
		In summary, we obtain 
		$\hat F_s(\xi)= (\sharp J) \hat f(\xi)$ for a.e. $\xi\in \Omega$,
		\begin{align}\label{FT_of_f1} 
			\hat f(\xi) = (\sharp J)^{-1} \hat F_s(\xi). 
		\end{align}

		%\begin{remark}\label{recovery}  
		%	Given a spectral pair $(A, J)$ in $\Bbb Z_N$, let $\Omega$ and $\Lambda$  be $\Omega:=  [0,1]+A$ and $\Lambda= \Bbb Z+J/N$. Let $f\in L^2(\Bbb R)$ with $\hat f$ supported in $\Omega$.    Then for any $f\in PW_\Omega$, the reconstruction formula holds: 
		%	$$\hat f(\xi) = (\sharp J)^{-1} \sum_{\lambda\in\Lambda} f(\lambda) e^{2\pi i \lambda \xi} \quad a.e. \ \xi \in \Omega.$$
		%	
		%	
		%\end{remark}

		%Furthermore, as mentioned above, we can write $F_s$ as 
		%$$ F_s(t) =\sum_{\lambda\in\Lambda} f(\lambda) \delta(t-\lambda) \quad \forall t\in \Bbb R, $$ 	and, 
		Note that by the definition of $F_s$ in (\ref{Fs-function1}), we obtain 
		$$\hat F_s(\xi) = \int_\Bbb R \left( \sum_{\lambda\in\Lambda} f(\lambda) \delta_\lambda(t) \right) e^{-2\pi i t \xi} dt = \sum_{\lambda\in\Lambda} f(\lambda) e^{-2\pi i \lambda \xi} \quad a.e. \ \xi \in \Omega.$$
		Substituting this in (\ref{FT_of_f1}) we complete the proof: 
		\begin{align}\notag%\label{FT_of_f_2}
			\hat f(\xi) = (\sharp J)^{-1} \sum_{\lambda\in\Lambda} f(\lambda) e^{-2\pi i \lambda \xi} \quad a.e. \ \xi\in \Omega.
		\end{align} 
		
	\end{proof}

\end{document}